\documentclass[11pt]{amsart}

\usepackage{amsmath, amssymb, amsthm, amsfonts, enumerate, color, comment, latexsym}

\usepackage{palatino}
\usepackage{graphicx}

\usepackage{parskip}

\numberwithin{equation}{section}
\theoremstyle{plain}
\newtheorem{Proposition}[equation]{Proposition}
\newtheorem{Corollary}[equation]{Corollary}
\newtheorem*{Corollary*}{Corollary}
\newtheorem{Theorem}[equation]{Theorem}
\newtheorem*{Theorem*}{Theorem}
\newtheorem{Lemma}[equation]{Lemma}
\theoremstyle{definition}

\newtheorem{Example}[equation]{Example}

\newtheorem*{Remark*}{Remark}

\usepackage{enumitem}
\setlist[enumerate]{leftmargin=*}
\setlist[itemize]{leftmargin=*}

\def\C{\mathbb{C}}
\def\R{\mathbb{R}}
\def\RR{\R}
\def\D{\mathbb{D}}
\def\T{\mathbb{T}}
\def\N{\mathbb{N}}
\def\Z{\mathbb{Z}}

\def\F{\mathcal {F}}

\def\Q{\mathcal{Q}}

\def\re{\mathop{\rm Re}\nolimits}
\def\Lap{\mathcal L} 

\renewcommand{\leq}{\leqslant}
\renewcommand{\geq}{\geqslant}
\renewcommand{\subset}{\subseteq}
\renewcommand{\phi}{\varphi}
\renewcommand{\vec}[1]{{\bf #1}}

\usepackage{xcolor}

\def\sech{\mathop{\rm sech}\nolimits}

\author[E. A. Gallardo-Guti\'errez]{Eva A. Gallardo-Guti\'errez}
\address{Eva A. Gallardo-Guti\'errez \newline
Departamento de An\'alisis Matem\'atico y Matem\'atica Aplicada,\newline
Facultad de Matem\'aticas,
\newline Universidad Complutense de
Madrid, \newline
Plaza de Ciencias 3, 28040 Madrid,  Spain
\newline
and Instituto de Ciencias Matem\'aticas ICMAT,
\newline Madrid,  Spain }
\email{eva.gallardo@mat.ucm.es}

\author[J. R. Partington]{Jonathan R. Partington}
\address{Jonathan R. Partington, \newline
School of Mathematics, \newline
University of Leeds, \newline
Leeds LS2 9JT, United Kingdom}
\email{J.R.Partington@leeds.ac.uk}

\author[W. T. Ross]{William T. Ross}
	\address{William T. Ross, \newline
	Department of Mathematics and Statistics, University of Richmond, Richmond, VA 23173, USA}
	\email{wross@richmond.edu}
	
\subjclass[2010]{Primary 47A15, 47A16, 47B15, 30H10}

\title[Hardy operators]{Hardy Operators: In the footsteps of Brown, Halmos, and Shields}

\date{\today}

\keywords{Hardy operator, Hardy space, inner function, Laguerre polynomials, Laguerre shift, frame vectors, commutant, Ces\`aro operator}

\thanks{First two authors are partially supported by Plan Nacional  I+D grant no. PID2022-137294NB-I00, Spain. First author also acknowledges support from
the Spanish Ministry of Science and Innovation, through the ``Severo Ochoa Programme for Centres of Excellence in R\&D'' (CEX2019-000904-S \& CEX2023-
001347) and from the Spanish National Research Council, through the ``Ayuda extraordinaria a Centros de Excelencia Severo Ochoa'' (20205CEX001).}

\begin{document}

\begin{abstract}
This paper discusses the two classical Hardy operators $\mathcal{H}_{1}$ on $L^2(0, 1)$ and $\mathcal{H}_{\infty}$ on $L^2(0, \infty)$ initially studied by Brown, Halmos and Shields. Particular  emphasis is given to the construction of explicit cyclic and $*$-cyclic vectors in conjunction with a characterization of their invariant and reducing subspaces. We also provide a complete description of the frame vectors for $I - \mathcal{H}_{1}^{*}$.
\end{abstract}

\maketitle

\section{Introduction}

This paper explores the invariant subspaces as well as the cyclic and frame vectors for the two classical Hardy integral operators. The first, known as the
{\em finite Hardy operator} $\mathcal{H}_1$, is defined on the Lebesgue space $L^2(0, 1)$ by the formula
$$(\mathcal{H}_1 f)(x) = \frac{1}{x} \int_{0}^{x} f(t) dt, \quad 0 < x < 1,$$ while the second, known as the {\em infinite Hardy operator} $\mathcal{H}_{\infty}$, is defined on $L^2(0, \infty)$ by
$$(\mathcal{H}_{\infty} g)(x) = \frac{1}{x} \int_{0}^{x} g(t) dt, \quad0 < x < \infty.$$
Although the formulas for these two operators are the same and have similar looking adjoint formulas
$$(\mathcal{H}_1^{*} f)(x) = \int_{x}^{1} \frac{f(t)}{t} dt, \quad 0 < x < 1,$$
$$(\mathcal{H}^{*}_{\infty} g)(x) = \int_{x}^{\infty} \frac{g(t)}{t}dt, \quad 0 < x < \infty,$$
and even have the same operator norm equal to $2$ (stemming from Hardy's integral inequality), they act boundedly on different Hilbert spaces and have distinct properties.

In 1965,  Brown, Halmos, and Shields \cite{MR187085} proved that the spectrum of $\mathcal{H}_1$ is the closed disk
$\{z: |z - 1| \leq 1\}$
while the spectrum of $\mathcal{H}_{\infty}$ is its boundary
$ \{z: |z - 1| = 1\}.$
A more salient difference between $\mathcal{H}_1$ and $\mathcal{H}_{\infty}$, also from \cite{MR187085}, comes from their representations as shifts,  namely   $I - \mathcal{H}_1^{*}$ is unitarily equivalent to the unilateral shift $(S f)(z) = z f(z)$ on the classical Hardy space of the open unit disk $\D := \{z: |z| < 1\}$ (and denoted by $H^2$), while $I - \mathcal{H}_{\infty}^{*}$ is unitarily equivalent to the bilateral shift $(M f)(\xi) = \xi f(\xi)$ on the Lebesgue space $L^2(\T, m)$, where $\T := \{z: |z| = 1\}$ is the unit circle and $dm = d\theta/(2 \pi)$ is normalized Lebesgue measure on $\T$.
As a straightforward consequence,  $\mathcal{H}_1^{*}$ is a subnormal operator, while $\mathcal{H}_{\infty}$ is normal.

Accordingly, the invariant subspace structure for the operators $\mathcal{H}_1$ and $\mathcal{H}_{\infty}$ is essentially understood since, in the first case, the invariant subspaces of $\mathcal{H}_1$ correspond, via Beurling's theorem on the invariant subspaces of the shift on $H^2$, to the model spaces $(\Theta H^2)^{\perp}: = H^2 \ominus \Theta H^2$ (where $\Theta$ is an inner function), while in the second, via results of Wiener and Helson on the invariant subspaces of the bilateral shift, the invariant subspaces of $\mathcal{H}_{\infty}$ correspond to orthogonal complements of spaces of the form $q H^2$ (where $q$ is a  unimodular function on $\T$) or $\chi_{E} L^2(\T, m)$ (where $E$ is a measurable subset of $\T$). See \cite[Ch.~II]{MR0171178} for the details on this.

More recently, Agler and McCarthy \cite{MR4671411} have shown that the invariant subspaces of ${\mathcal H}_1$ can also be expressed as limits of sequences of finite dimensional spaces spanned by eigenfunctions for $\mathcal{H}_1$, namely the monomial functions $x \mapsto x^s$ for $s \in \C$ with $\re s > -\frac{1}{2}$, as so--called ``liminf spaces''.
This is in contrast to the related Volterra operator
$$f(x) \mapsto \int_{0}^{x} f(t) dt$$ on $L^2(0,1)$, which has no finite dimensional invariant subspaces (see, for example, \cite[Chap.~7]{MR4545809}). Indeed, every invariant subspace of the Volterra operator  takes the form $\chi_{[a, 1]} L^2(0, 1)$ for some $0 \leq a \leq 1$.

The above descriptions of the invariant subspaces of $\mathcal{H}_1$ and $\mathcal{H}_{\infty}$ rely either on the unitary operators involved in  identifying $I - \mathcal{H}_1^{*}$ with $S$ and $I - \mathcal{H}_{\infty}^{*}$ with $M$ or, in the somewhat implicit approach of Agler and McCarthy, as limits of functions via ``liminf spaces''. This paper starts  by giving an \emph{explicit} description of the invariant subspaces in terms of the functions in the spaces $L^2(0, 1)$ or $L^2(0, \infty)$ and not their realizations on $H^2$ or $L^2(\T, m)$. The ideas of Brown, Halmos and Shields, along with an explicit inversion formula for the Mellin transform  (Proposition \ref{inversion lemma}), allows us to characterize the reducing subspaces for $\mathcal{H}_{\infty}$ (i.e., those invariant for both $\mathcal{H}_{\infty}$ and $\mathcal{H}_{\infty}^{*}$) as
$$ \left\{t \mapsto \int_{E} t^{- i x - \frac{1}{2}} f(ix) dx: f \in L^2(i\R)\right\}$$
for some measurable subset $E \subset \R$ (see Theorem \ref{reducingeweswefrT}). Note that $\mathcal{H}_1$ lacks reducing subspaces. This description is explicit in a concrete sense: whether a given $f \in L^2(0,\infty)$ belongs to such a subspace can be decided from $f$ itself, by computing a single integral, namely its Mellin transform $$\int_0^\infty f(t)\, t^{i\beta - \frac12}\, dt$$ along the line $\re s = \frac12$, and checking where it vanishes. No passage to another space is needed.

This, in turn, leads us to a characterization in Theorem \ref{prop:111} of the $\ast$-cyclic vectors for $\mathcal{H}_{\infty}$, i.e., those $f \in L^2(0, \infty)$ not contained in any nontrivial reducing subspace. Specifically, $f \in L^2(0, \infty)$ is $\ast$-cyclic for $\mathcal{H}_{\infty}$ if and only if  the function $$\beta \mapsto \int_{-\infty}^{\infty} f(e^{x}) e^{(i \beta + \frac{1}{2}) x}dx$$ does not vanish on any  subset of $\R$ with positive measure. Some particular classes (Corollary \ref{supporttstsaaa}) of $\ast$-cyclic vectors include all nontrivial $L^2(0,\infty)$ functions supported on $[0,a]$ or supported on  $[a,\infty)$, $0 < a < \infty$.

The study of the cyclic vectors (vectors not contained in any nontrivial  invariant subspace) for $\mathcal{H}_1^{*}$ and $\mathcal{H}_{\infty}^{*}$ is a complex undertaking as the provided criteria are difficult to verify. As particular instances, we show in  \S \ref{cyclkcic5} that the functions $t \mapsto t^{\alpha}$, where $\re \alpha > -\frac{1}{2}$, are cyclic vectors for $\mathcal{H}_1^{*}$ while the functions
$$\frac{1}{(t + 1)^n},\;  n  = 1, 2, \ldots,  \; \;  \frac{1}{\sqrt{t} (1 + (\log t)^2)}, \;  \; \mbox{and} \; e^{-t}$$
are cyclic for  both $\mathcal{H}_{\infty}$ and $\mathcal{H}^*_{\infty}$. In fact, $\mathcal{H}_{\infty}$ and $\mathcal{H}_{\infty}^{*}$ share the same cyclic vectors (Proposition \ref{share}).

This paper will also make connections to semigroups of composition operators on $L^2(0, \infty)$, the Laguerre shift, and, using our Mellin transform analysis, the commutants of  $\mathcal{H}_1$ and $\mathcal{H}_{\infty}$. These results will also help demonstrate the variety of interesting invariant subspaces for our two Hardy operators (see Examples \ref{bhsdf7gsdfds7f} and \ref{100ooOO00}).

The final section of the paper follows the work by Cabrelli, Molter, Paternostro, and Philipp \cite{MR4093918}
and characterizes the frame vectors for  a particular variation of $\mathcal{H}_1$. Here, for a bounded linear operator $T$ on separable Hilbert space $H$, a vector $\vec{v} \in H$ is a {\em frame vector} for $T$ if the sequence $(T^n \vec{v})_{n \geq 0}$ is a frame in that there are positive constants $c_1, c_2$ so that
$$c_1 \|\vec{x}\|^2 \leq \sum_{n = 0}^{\infty} |\langle T^n \vec{v}, \vec{x}\rangle|^2 \leq c_2 \|\vec{x}\|^2 \; \; \mbox{for all $\vec{x} \in H$}.$$ In particular,
\cite[Thm.~3.3]{MR4093918} asserts that for $T$ to have frame vectors at all,
it is necessary that $T^*$ is similar to a strongly stable contraction, in other words,
$$\lim_{n \to \infty} \|T^{*n}\vec{y}\| = 0 \; \; \mbox{ for all $\vec{y} \in H$.}$$
By the uniform boundedness principle, it follows that such $T$ must satisfy
$$\sup_{n \geq 0} \|T^n\| < \infty$$
and hence, in particular,  $\sigma(T) \subseteq \overline{\D} := \{z: |z| \leq 1\}$.
Thus, we cannot hope to find frame vectors for the Hardy operators $\mathcal{H}_1$ and $\mathcal{H}_{\infty}$, since, as just mentioned,  their spectra are not contained
in $\overline{\D}$. However, we will show in \S \ref{section frame} that $I-\mathcal{H}_1^*$ does possess frame vectors while neither $I-\mathcal{H}_1$ nor $I-\mathcal{H}_{\infty}$ do.
Moreover, frame vectors for $I - \mathcal{H}_{1}^{*}$ exist in abundance and we give plenty of specific examples as well as a general characterization (Corollary \ref{generalframeHH}). Along the way, we show there are no frame vectors for $I - \mathcal{C}$, where $\mathcal{C}$ is the classical (discrete) Ces\`{a}ro operator on $\ell^2$.

\section{Preliminaries}

In this section, we collect some preliminaries regarding the framework of this manuscript. In particular, we will prove Proposition \ref{inversion lemma} which will play a key role in our \emph{explicit} description of the invariant subspaces of the Hardy operators $\mathcal{H}_1$ and $\mathcal{H}_\infty$ on $L^2(0, 1)$ and $L^2(0, \infty)$, respectively.

\smallskip

An important tool in our analysis will involve $H^2(\C^{+})$, the {\em Hardy space of the right half plane}
$\C^{+}: = \{z: \re z > 0\}$. Here an analytic function  $f$ on $\C^{+}$ belongs to $H^2(\C^{+})$ when it satisfies the bounded integral means condition
$$\sup_{x > 0} \int_{-\infty}^{\infty} |f(x + i y)|^2 dy < \infty.$$
The square root of the above quantity defines the norm of $f$ and will be denoted by $\|f\|_{H^2(\C^{+})}$. Well known theory about Hardy spaces \cite{Garnett,  MR3890074, MR2158502} says that for each $f \in H^2(\C^{+})$, the ``radial limit''
$$f(iy) := \lim_{x \to 0^{+}} f(x + i y)$$ exists for almost every $y \in \R$ and this function belongs to $L^2(i\R)$ with
$$\|f\|_{H^2(\C^{+})}^{2}  = \int_{-\infty}^{\infty} |f(i y)|^2 dy.$$
We will use the notation  $H^{2}_{+}(i\R)$ for the set of radial boundary functions for $H^2(\C^{+})$ (which will be a closed subspace of $L^2(i\R)$). In a similar way, we  define the {\em Hardy space of the left half plane} $\C^{-} := \{z: \re z < 0\}$ consisting of all analytic functions $g$ on $\C^{-}$ with
$$\sup_{x < 0} \int_{-\infty}^{\infty} |g(x + i y)|^2 dy < \infty$$ with analogously defined boundary functions (with the radial limit as $x \to 0^{-}$). We denote the corresponding space of radial boundary functions by $H^{2}_{-}(i\R)$. One can verify that the linear transformation $C: L^2(i \R) \to L^2(i \R)$ defined by $(C f)(i x) = f(-ix)$ is a unitary operator  (and its own adjoint)  and satisfies
\begin{equation}\label{Ceeee}
C H^{2}_{+}(i \R) = H^{2}_{-}(i \R).
\end{equation}

\smallskip

The Paley--Wiener theorem (see \cite[p.~146]{MR1864396} or  \cite[p.~203]{MR3890074}) says that the (normalized) \emph{Laplace transform}
\begin{equation}\label{LTT}
(\mathcal{L} f)(z) = \frac{1}{\sqrt{2 \pi}} \int_{0}^{\infty} e^{-u z} f(u) du, \quad z \in \C^{+},
\end{equation} is an isometric isomorphism from $L^2(0, \infty)$ onto $H^2(\C^{+})$. Moreover,  the corresponding Laplace transform
$$\frac{1}{\sqrt{2 \pi}} \int_{-\infty}^{0} e^{- u z} f(u) du, \quad z \in \C^{-},$$ is an isometric isomorphism from $L^2(-\infty, 0)$ onto $H^2(\C^{-})$. From here it follows that
\begin{equation}\label{***}
L^2(i \R) = H^{2}_{+}(i \R) \oplus H^{2}_{-}(i \R).
\end{equation}

\smallskip

Also important for our analysis will be the (modified) {\em Mellin transform}.  It is known (see \cite[p.~166]{MR1864396} or  \cite[p.~204]{MR3890074}) that the linear transformation
$$f \mapsto  \frac{1}{\sqrt{2 \pi}} \int_{0}^{\infty} f(x) x^{is - \frac{1}{2}}dx, \quad -\infty < s < \infty,$$
defines an isometric  isomorphism between $L^2(0, \infty)$ and $L^2(\R)$. Thus, the operator $\mathcal{Q}$ defined by
$$(\mathcal{Q} f)(i \beta) = \frac{1}{\sqrt{2 \pi}} \int_{0}^{\infty} f(t) t^{i \beta - \frac{1}{2}} dt, \qquad -\infty<\beta<\infty,$$
is an isometric isomorphism between $L^2(0, \infty)$ and $L^2(i \R)$.

The Mellin transform $\mathcal{Q}$ is often difficult to compute and even more difficult to invert. To help us along, we will make use of the following connection (implicitly discussed in \cite[p.~166]{MR1864396}) between $\mathcal{Q}$ and the Fourier transform $\mathcal{F}$ on $L^2(\R)$. Recall that $\mathcal{F}$ is initially defined on $L^1(\R)$ by
$$(\mathcal{F} f)(\alpha) = \frac{1}{\sqrt{2 \pi}} \int_{-\infty}^{\infty} e^{-i\alpha x} f(x) dx$$
and extended to be a unitary operator on $L^2(\R)$ with
$$(\mathcal{F}^{-1} f)(x) = \frac{1}{\sqrt{2 \pi}} \int_{-\infty}^{\infty} e^{i x \alpha} f(\alpha)d \alpha.$$

Here is an important tool to help invert $\mathcal{Q}$. We first introduce the isometric isomorphism $T: L^2(0,\infty) \to L^2(\R)$ defined by
\begin{equation}\label{TTTTisoisoms}
(Tf)(x)=f (e^x) e^{\frac{x}{2}}, \quad - \infty < x < \infty,
\end{equation}
noting, by the change of variables $t=e^x$, that
\[
\|Tf\|_{L^2(\R)}^2 = \int_{-\infty}^\infty |f(e^x)|^2 e^x \, dx = \int_{0}^\infty |f(t)|^2 \, dt = \|f\|_{L^2(\R)}^2.
\]
 One can then invert $T$ with the formula
\begin{equation}\label{TTTTinverse}
(T^{-1} g)(t)=\frac{g(\log t)}{\sqrt{t}}, \quad 0 < t < \infty.
\end{equation}

This yields  the following useful identity (more general results can be found in \cite{MR942661},
for example).

\begin{Proposition}\label{inversion lemma}
For each $f \in L^2(0,\infty)$, we have
$$(\mathcal{Q} f)(i\beta)=(\mathcal{F}^{-1} Tf )(\beta)$$  for almost every $\beta \in \R$.
\end{Proposition}

\begin{proof}
By the same change of variables as above,
\begin{align*}
(\mathcal{Q} f)(i\beta) &= \frac{1}{\sqrt{2\pi}} \int_{-\infty}^\infty f(e^x) e^{(i\beta-\frac{1}{2})x}e^x \, dx\\
&= \frac{1}{\sqrt{2\pi}} \int_{-\infty}^\infty  e^{i\beta x} (Tf)(x) \, dx \\
&= (\mathcal{F}^{-1} T f) (\beta),
\end{align*}
which yields the desired formula.
\end{proof}

Thus, using \eqref{TTTTinverse},  $\mathcal{Q}^{-1} : L^2(i\R) \to L^2(0, \infty)$ is given by
\begin{align}\label{Mellintrack}
(\mathcal{Q}^{-1} g)(t) & =(T^{-1}\mathcal{F} g(i \cdot))(t)\\\nonumber
& = T^{-1} \Big(y \mapsto \frac{1}{ \sqrt{2 \pi}} \int_{-\infty}^{\infty} e^{-ixy} g(ix) dx\Big)(t)\\\nonumber
& = \frac{1}{\sqrt{2 \pi}} \frac{1}{\sqrt{t}} \int_{-\infty}^{\infty} e^{-ix \log t} g(ix) dx\\\nonumber
& = \frac{1}{\sqrt{2 \pi}} \frac{1}{\sqrt{t}} \int_{-\infty}^{\infty} t^{-ix} g(ix)dx.
\end{align}

In a similar way as above, the isometric isomorphism
\begin{equation}\label{WWWWWW}
W: L^2(0, 1) \to L^2(0, \infty), \quad (W f)(u) = f(e^{-u}) e^{-\frac{u}{2}}
\end{equation}
satisfies
\[
(W^{-1}g) (t)=\frac{g(-\log t)}{\sqrt{t}}.
\]
Thus, for $f \in L^2(0, 1)$, a use of the Paley--Wiener theorem says that the function
$$z \mapsto \int_{0}^{\infty} (W f)(u) e^{-u z} du$$ on $\C^{+}$ belongs to $H^2(\C^{+})$. Moreover,
\begin{align}\label{sd98gsdf1111100}
 \int_{0}^{\infty} (W f)(u) e^{-u z} du & = \int_{0}^{\infty} f(e^{-u}) e^{-\frac{u}{2}} e^{-u z}du\\\nonumber
 & = \int_{0}^{1} f(t) t^{z - \frac{1}{2}} dt.
\end{align}
This shows that
\begin{equation}\label{s0f8oigfjdkafdWW}
\mathcal{Q} f = \mathcal{L} W f \; \; \mbox{for all $f \in L^2(0, 1)$},
\end{equation}
and thus
$$\mathcal{Q} L^2(0, 1) = H^{2}_{+}(i\R).$$
Since $H^{2}_{-}(i \R) = H^{2}_{+}(i \R)^{\perp}$ by \eqref{***} and $\mathcal{Q}$ is an isometric isomorphism, we also see that
\begin{equation}\label{kkKKhhKK}
\mathcal{Q} L^2(1, \infty) = H^{2}_{-}(i \R).
\end{equation}

\section{The Hardy operators as multiplication operators}\label{sectiosn333}

In this section, we realize $I - \mathcal{H}_{\infty}^{*}$ as a multiplication operator on $L^2(i\R)$. This was implicitly done in the Brown, Halmos, Shields paper \cite{MR187085} but we make this realization  explicit here since we need it  to appear in a certain form. The main results of this section are Proposition \ref{98rgegbrfewrfvFF} and Proposition \ref{yyysysSYY}, but we need some preliminary results along the way.  In order to do this, we recall the Laguerre polynomials.

\subsection{The Laguerre polynomials} The {\em Laguerre polynomials} $(L_n)_{n \geq 0}$ are a sequence of polynomials which serve as an orthonormal basis for the weighted $L^2$ space $L^2((0, \infty), e^{-x}dx)$ in that
$$\int_{0}^{\infty} L_{m}(x) L_{n}(x) e^{-x} dx = \delta_{mn}$$ and $(L_n)_{n \geq 0}$ is complete in $L^2((0, \infty), e^{-x}dx)$. See \cite[Ch.~V]{MR372517} for the details.
Letting
$$V = I - \mathcal{H}_{\infty}^{*},$$ one can show by a direct integral calculation that $V$ is a unitary operator on $L^2(0, \infty)$. Here is a calculation which follows from properties of Laguerre polynomials  \cite{Kruglyak2006StructureOT}.

\begin{Lemma}\label{0a9srfgrdsvcxsd}
If $\chi$ denotes the characteristic function for $(0, 1)$, we have the following
\begin{enumerate}
\item[(a)] For each $n = 0, 1, 2, \ldots,$
$$(V^{n} \chi)(x) = L_{n}(-\log x) \chi(x)$$
and $(V^{n} \chi)_{n \geq 0}$ forms an orthonormal basis for $L^2(0, 1)$.
\item[(b)] For each $n = 0, 1, 2, 3 \ldots,$
$$(V^{* {{n+1}}} \chi)(x) = - \frac{L_{n}(\log x)}{x} (1 - \chi(x))$$
and $(V^{*n} \chi)_{n \geq 1}$ forms an orthonormal basis for $L^2(1, \infty)$.
\end{enumerate}
\end{Lemma}

Now consider the  function
\begin{equation}\label{phiiiii}
\phi(s) = \frac{s - \tfrac{1}{2}}{s + \tfrac{1}{2}}
\end{equation}
and observe that $\phi$ maps the right half plane  $\C^{+}$ conformally onto $\D$ and $|\phi(iy)| = 1$ for all $y\in\R$.
This induces an isometric isomorphism $J: L^2(\T, m) \to L^2(i\R)$ defined by
\begin{equation}\label{eq:Jh2h2}
(Jf)(s)= \frac{1}{\sqrt{2 \pi}} \frac{1}{s+\frac12} f(\phi(s)), \quad s \in i\R,
\end{equation}
which restricts to an  isometric isomorphism from $H^2$ (the Hardy space of the disk) onto $H^2(\C^+)$ (the Hardy space of the right half plane). Observe that  the multiplication operator
\begin{equation}\label{mspshsiii}
M_{\phi}: L^2(i\R) \to L^2(i\R), \quad M_{\phi} f = \phi f,
\end{equation} is unitary and, moreover, $M_{\phi} H^{2}_{+}(i\R) \subset H^{2}_{+}(i\R)$. A calculation will verify that
$$J^{-1} M_\phi J$$ is the bilateral shift on $L^2(\T, m)$, which allows us to
carry the Beurling--Lax and Wiener invariant subspace theorems over to
$L^2(i\R)$ and $H^2(\C^+)$.

Using the well-known fact that $(z^n)_{n \geq 0}$ forms an orthonormal basis for $H^2$, along with the fact that $J$ is an isometric isomorphism, we see that the functions
\begin{equation}\label{s09dfgioaVHFJKRFG}
u_{n}(i \beta) :=  \frac{1}{\sqrt{2 \pi}} \Big(\frac{i \beta - \frac{1}{2}}{i \beta + \frac{1}{2}}\Big)^{n} \frac{1}{i \beta + \frac{1}{2}}, \quad n = 0, 1, 2, \cdots,
\end{equation}
form an orthonormal basis for $H_{+}^2(i \R)$ and that $M_{\phi} u_n = u_{n + 1}$.
The following connects  $\mathcal{Q}$ with the Laguerre polynomials and the orthonormal basis $(u_n)_{n \geq 0}$.

The Laplace transform formula yields that
\[
\int_0^\infty e^{-xs} x^n \, dx = \frac{n! }{s^{n+1}}, \quad s \in \C^{+},
\]
and,
since the Laguerre polynomials $L_{n}(x)$ for $n = 0, 1, 2, \ldots,$ can be written as
\begin{equation}\label{eq:Ln}
L_n(x)= \frac{1}{n!} \left( \frac{d}{dx}-1\right)^n x^n,\
\end{equation}
we conclude   that
\[
\int_0^\infty e^{-xs} L_n(x) \, dx =  \frac{(s-1)^n }{s^{n+1}}
\]
and so
\[
\int_0^\infty e^{-xs} e^{-\frac{x}{2}} L_n(x) \, dx =  \frac{(s-\tfrac{1}{2})^n}{(s+\tfrac{1}{2})^{n+1}}.
\]
We may rewrite this as a Fourier transform, taking $s=i\omega$, and then
\begin{equation}\label{eq:ftln}
\frac{1}{\sqrt{2\pi}}\int_0^\infty e^{-i\omega x} e^{-\frac{x}{2}} L_n(x) \, dx = \frac{1}{\sqrt{2\pi}} \frac{(i\omega-\tfrac{1}{2})^n }{(i\omega+\tfrac{1}{2})^{n+1}}.
\end{equation}

\begin{Proposition}\label{prop:QLn}
For each $n = 0, 1, 2, \ldots,$ we have
$$(\mathcal{Q} L_{n}(- \log x) \chi )(i \beta)= u_n(i\beta).$$
\end{Proposition}

\begin{proof}
An integral substitution in $(\mathcal{Q} L_{n}(- \log x) \chi )(i \beta)$ will yield
$$(\mathcal{Q} L_{n}(- \log x) \chi )(i \beta) = \frac{1}{\sqrt{2 \pi}} \int_{0}^{\infty} e^{-i\beta u} e^{-\frac{u}{2}} L_{n}(u) du.$$
Now use \eqref{eq:ftln}.
\end{proof}

Furthermore, as observed earlier in Lemma \ref{0a9srfgrdsvcxsd}, noting that
\begin{equation}\label{redaskk}
(I - \mathcal{H}_{\infty}^{*})|_{L^2(0, 1)} = I - \mathcal{H}_1^{*},
\end{equation}
we see that
$$(I - \mathcal{H}_1^{*}) L_{n}(-\log x) \chi = L_{n + 1}(-\log x) \chi,$$ and thus
\begin{equation}\label{eq:h1equiv}
\mathcal{Q} (I - \mathcal{H}_1^{*}) f = M_{\phi} \mathcal{Q} f, \quad f \in L^2(0, 1).
\end{equation}
We summarize this with the following proposition.

\begin{Proposition}\label{98rgegbrfewrfvFF}
$I - \mathcal{H}_1^{*}$ on $L^2(0, 1)$ is unitarily equivalent to $M_{\phi}$ on $H^2_{+}(i\R)$.
\end{Proposition}

Using the substitution $z \mapsto - z$, which transforms $H^2(\C^{+})$ onto $H^2(\C^{-})$ (recall \eqref{Ceeee}), we conclude that the functions
\begin{equation}\label{vvvveee}
v_{n}(i\beta) := i \sqrt{\frac{2}{\pi}} \Big(\frac{2 \beta - i}{2 \beta + i}\Big)^{n} \frac{1}{2 \beta + i}, \quad n = 0, 1, 2, \ldots,
\end{equation} form an orthonormal basis for $H_{-}^{2}(i\R)$. Another connection that the Mellin transform  $\mathcal{Q}$ makes with the Laguerre polynomials is the following.

\begin{Proposition}
For each $n = 0, 1, 2, \ldots$ we have
$$(\mathcal{Q} (-\frac{ L_{n}(\log x)}{x} (1 - \chi)))(i \beta) = -v_{n}(i \beta).$$
\end{Proposition}

Finally, notice that or $n\geq 1$ (from Lemma ~\ref{0a9srfgrdsvcxsd})
\begin{align*}
\mathcal{Q} (I - \mathcal{H}_{\infty}^{*}) \mathcal{Q}^{-1} v_{n}(i \beta) & =
\mathcal{Q} (I - \mathcal{H}_{\infty}^{*}) (\frac{ L_{n}(\log x)}{x} (1 - \chi))(i\beta)\\
 & = \mathcal{Q} (\frac{L_{n -1}(\log x)}{x} (1 - \chi))(i\beta)\\
& = i \sqrt{\frac{2}{\pi}} \Big(\frac{2 \beta - i}{2 \beta + i}\Big)^{n - 1} \frac{1}{2 \beta + i}\\
& = M_{\phi}  i \sqrt{\frac{2}{\pi}} \Big(\frac{2 \beta - i}{2 \beta + i}\Big)^{n} \frac{1}{2 \beta + i}\\
& = M_{\phi} v_{n}(i \beta).
\end{align*}
For $n=0$ the image of $\mathcal{Q}^{-1}v_0$ under $I - \mathcal{H}_{\infty}^{*}$ no longer lies in $L^2(1,\infty)$, so this case has to be checked separately. By the previous proposition and Lemma~\ref{0a9srfgrdsvcxsd}(b),
$$\mathcal{Q}^{-1} v_0 = \frac{1}{x}(1-\chi) = -V^{*}\chi,$$
and hence, since $V$ is unitary, $(I - \mathcal{H}_{\infty}^{*})\mathcal{Q}^{-1} v_0 = -VV^{*}\chi = -\chi$. Since $\chi = L_0(-\log x)\chi$, Proposition~\ref{prop:QLn} gives
\begin{align*}
\mathcal{Q} (I - \mathcal{H}_{\infty}^{*}) \mathcal{Q}^{-1} v_{0}(i \beta) & = -u_0(i\beta)\\
& = -\frac{1}{\sqrt{2\pi}}\, \frac{1}{i\beta+\frac12}\\
& = i \sqrt{\frac{2}{\pi}}\, \frac{1}{2\beta - i}.
\end{align*}
On the other hand, $\phi(i\beta) = \dfrac{2\beta + i}{2\beta - i}$, so that
$$M_{\phi} v_0(i\beta) = i \sqrt{\frac{2}{\pi}}\, \frac{2\beta+i}{2\beta-i}\, \frac{1}{2\beta+i} = i \sqrt{\frac{2}{\pi}}\, \frac{1}{2\beta - i}.$$
Therefore $\mathcal{Q} (I - \mathcal{H}_{\infty}^{*}) \mathcal{Q}^{-1} v_{n} = M_{\phi} v_{n}$ for every $n \geq 0$.

Thus, since $(v_n)_{n \geq 0}$ is an orthonormal basis for $H^{2}_{-}(i \R)$ and $\mathcal{Q} L^2(1, \infty) = H^{2}_{-}(i \R)$ (see \eqref{kkKKhhKK}), it follows that
$$\mathcal{Q} (I - \mathcal{H}_{\infty}^{*}) f = M_{\phi} \mathcal{Q} f, \quad f \in L^2(1, \infty).$$ Combining this with \eqref{redaskk} and \eqref{eq:h1equiv}, we have the following.

\begin{Proposition}\label{yyysysSYY}
$I - \mathcal{H}_{\infty}^{*}$ on $L^2(0, \infty)$ is unitarily equivalent to $M_{\phi}$ on $L^2(i\R)$.
\end{Proposition}

The unitary operator $f(i x) \mapsto f(-ix)$ on $L^2(i \R)$ induces a unitary equivalence between $M_{\phi}$ and $M_{\overline \phi} = M^{*}_{\phi}$. This yields the following corollary.

\begin{Corollary}
$I - \mathcal{H}_{\infty}$ on $L^2(0, \infty)$ is unitarily equivalent to $M_{\phi}$ on $L^2(i \R)$.
\end{Corollary}

\subsection{A connection to the Laguerre shift}

The {\em Laguerre shift} was formulated by von Neumann  \cite{MR1581206} (see also \cite[p.~18]{MR822228}) as the bounded operator $\mathfrak{L}$ on $L^2(0, \infty)$ defined by
\begin{equation}\label{eq:vonNeumann}
(\mathfrak{L} f)(x) = f(x) - \int_{0}^{x} e^{\frac{t - x}{2}} f(t) dt.
\end{equation}
Recalling our earlier discussion that the \emph{Laguerre functions} defined by
$$\Phi_{n}(x) = e^{-\frac{x}{2}} L_{n}(x), \quad n = 0, 1, 2, \ldots,$$
constitute an orthonormal basis for $L^2(0, \infty)$,
von Neumann showed that
\begin{equation}\label{eq:lagshiftidentity}
\mathfrak{L} \Phi_n = \Phi_{n + 1}
\end{equation}
and thus $\mathfrak{L}$ is a shift of multiplicity one. General theory of such shifts shows that $\mathfrak{L}$ is unitarily equivalent to the unilateral shift $(S f)(z) = z f(z)$ on $H^2$.

Recalling the fact that
$$\frac{1}{\sqrt{2\pi}} \frac{(i\omega- \frac{1}{2})^n }{(i\omega+ \frac{1}{2})^{n+1}} = u_{n}(i \omega), \quad n = 0, 1, 2, \ldots $$
is the orthonormal basis for $H^{2}_{+}(i\R)$ discussed in \eqref{s09dfgioaVHFJKRFG}, the identity in \eqref{eq:ftln} says that
$$\mathcal{L} \mathfrak{L} \Phi_{n}=  M_{\phi} \mathcal{L} \Phi_n,$$
where $\mathcal{L}$ is the (normalized) Laplace transform from \eqref{LTT}. This provides a restatement of \eqref{eq:lagshiftidentity} in the following way.

\begin{Proposition}\label{prop unitary}
The Laguerre shift $\mathfrak{L}$ on $L^2(0, \infty)$ is unitarily equivalent to $M_{\phi}$ on $H^{2}_{+}(i \R)$.
\end{Proposition}

Recall that a vector $\vec{v}$ in a separable Hilbert space $H$ is {\em cyclic} for a bounded linear operator $T$ on $H$ if $\vec{v}$ is not contained in any nontrivial invariant subspace for $T$, equivalently,  the linear span of its orbit $(T^n \vec{v})_{n \geq 0}$ is dense in $H$. One  can straightforwardly use Proposition \ref{prop unitary} to provide some cyclic vectors for the Laguerre shift. Namely, if $f \in H^2(\C^+)$ is outer and hence, by Beurling's Theorem, a cyclic vector for $M_{\phi}$, then $\mathcal{L}^{-1} f$ will be cyclic for the Laguerre shift. Particular instances of cyclic vectors for $\mathfrak{L}$ are the $L^2(0, \infty)$ functions
$$x^n e^{-c x}, \quad \re c > 0,  \quad n = 0, 1, 2, \ldots$$
since  their Laplace transforms will be nonzero constants times $(z + c)^{-n-1}$, which are outer functions in $H^2(\C^{+})$.
Examples of different nature will be provided in \S \ref{cyclkcic5}.

\subsection{A connection to the bilateral Laguerre shift}

The previous discussion can be taken a bit further by considering negative indices as well. Here an orthonormal
basis for $L^2(-\infty,0)$ is clearly given by
$$-e^{\frac{x}{2}} L_m(-x)\chi_{(-\infty,0)}(x), \quad m = 0, 1, 2, \ldots,$$ and, similarly as before,  these functions
Fourier transform to
$$\dfrac{1}{\sqrt{2\pi}} \dfrac{(i\omega+\frac{1}{2})^m }{(i\omega-\frac{1}{2})^{m+1}}, \quad m = 0, 1, 2, \ldots.$$
From \eqref{vvvveee}, this  is an orthonormal basis for $H^2_-(i\R)$.
A good reference for the bilateral Laguerre shift is \cite{MR611399}.

The   Laguerre shift in the transform domain is just multiplication by the function $\phi(i \omega)$ from \eqref{phiiiii}.
We have an orthonormal basis for $L^2(\R)$ defined by
\[
\Phi_{n}(x) = \begin{cases}
-e^{\frac{x}{2}} L_{-n-1}(-x) \chi_{(-\infty,0)}(x),& n = -1,-2, -3, \ldots\\
e^{-\frac{x}{2}} L_{n}(x)\chi_{(0,\infty)}(x),& n = 0, 1, 2, \ldots.\\
\end{cases}
\]

The bilateral Laguerre shift again satisfies $\mathfrak{L} \Phi_n= \Phi_{n+1}$, only now for all $n \in \Z$ (and not just $\N_{0}$ as before).
From these considerations we see that $\mathfrak{L}$ is unitarily equivalent to $M_{\phi}$ on $L^2(i \R)$.
Moreover, the von Neumann formula \eqref{eq:vonNeumann} extends to  $L^2(\R)$
by
\[
(\mathfrak{L} f)(x) = f(x) - \int_{-\infty}^{x} e^{\frac{t - x}{2}} f(t) dt,
\]
as can be verified by the analysis above.

\section{Invariant  subspaces for $\mathcal{H}_{\infty}$}

From Proposition \ref{yyysysSYY} and Proposition~\ref{inversion lemma} it follows that every invariant subspace for $\mathcal{H}_{\infty}^{*}$ must take the form
\begin{equation}\label{mmmMMminnversee}
\mathcal{Q}^{-1} \mathcal{M}  = T^{-1} \mathcal{F} \mathcal{M},
\end{equation}
where $\mathcal{M}$ is an $M_{\phi}$-invariant subspace of $L^2(i \R)$. Recall the multiplication operator $M_{\phi}$ from \eqref{mspshsiii}.

Of particular interest among the invariant subspaces of a linear bounded operator $T$ acting on a Hilbert space $H$ are the reducing ones since they encode the projections in the commutant of the operator. Recall that a non-trivial subspace $M$ is a \emph{reducing subspace} if both $M$ and its orthogonal complement $M^\perp$ are invariant under $T$, or, equivalently, the orthogonal projection $P_M$ onto $M$ commutes with $T$.

\subsection{Reducing subspaces for $\mathcal{H}_{\infty}$}

Well known results of Wiener (originally stated for $L^2$ of the circle, see  \cite[Ch.~2]{MR0171178}, but which can be stated for $L^2(i\R)$ via the unitary operator $J$ from \eqref{eq:Jh2h2}), describe the  reducing subspaces for $M_{\phi}$ on $L^2(i\R)$ as
$\chi_{iE} L^2(i\R)$ for some (Lebesgue) measurable subset $E \subset \R$. This gives us the following.

\begin{Theorem}\label{reducingeweswefrT}
A subspace $\mathcal{N}$ of $L^2(0, \infty)$ is reducing for $\mathcal{H}_{\infty}$ if and only if there is a measurable subset $E$ of $\R$ such that
$$\mathcal{N} =  \left\{t \mapsto \int_{E} t^{- i x - \frac{1}{2}} f(ix) dx: f \in L^2(i\R)\right\}.$$
\end{Theorem}

\begin{proof}
From \eqref{mmmMMminnversee} we see that $\mathcal{N}$ is reducing for $\mathcal{H}_{\infty}$ if and only if
there is a measurable subset $E$ of $\R$ such that
$$\mathcal{N} = \mathcal{Q}^{-1} \chi_{i E} L^2(i\R)  = T^{-1} \mathcal{F}  \chi_{i E} L^2(i\R).$$ Then from \eqref{Mellintrack} we have
\begin{align*}
\mathcal{Q}^{-1} \chi_{i E} L^2(i\R)
& = \left\{t \mapsto \int_{E} t^{- i x - \frac{1}{2}} f(ix) dx: f \in L^2(i\R)\right\}. \qedhere
\end{align*}
\end{proof}

Here, and in Theorem~\ref{prop:111}, \eqref{eq:redhstar}, \eqref{eq:nonreducingh} and \S\ref{cyclkcic5}, the normalizing constant $1/\sqrt{2\pi}$ in \eqref{Mellintrack} is absorbed into the function $f$ (or $h$), since this does not change the subspaces involved.

\begin{Remark*}
Theorem~\ref{reducingeweswefrT} is the first payoff of the explicit approach announced in the introduction. Wiener's theorem, combined with the unitary equivalence of Brown, Halmos and Shields, only describes the reducing subspaces of $\mathcal{H}_{\infty}$ as images of the spaces $\chi_{F} L^2(\T, m)$ under an unspecified unitary map. Theorem~\ref{reducingeweswefrT} instead describes them directly as spaces of functions on $(0,\infty)$, and makes the following explicit.
\begin{enumerate}
\item[(a)] \emph{Membership.} A function $f \in L^2(0,\infty)$ lies in the reducing subspace $\mathcal{N}_E$ associated with $E$ if and only if
$$\int_0^\infty f(t)\, t^{i\beta-\frac12}\, dt = 0 \quad \text{for almost every } \beta \in \R \setminus E.$$
\item[(b)] \emph{Projections.} The orthogonal projection onto $\mathcal{N}_E$ is $P_{\mathcal{N}_E} = \mathcal{Q}^{-1} M_{\chi_{iE}} \mathcal{Q}$, that is,
$$(P_{\mathcal{N}_E} f)(t) = \frac{1}{2\pi} \int_{E} t^{-ix-\frac12} \Big(\int_0^\infty f(u)\, u^{ix-\frac12}\, du\Big) dx.$$
This is the case $g = \chi_{iE}$ of Proposition~\ref{Commutant H infty}.
\item[(c)] \emph{The lattice of reducing subspaces.} We have $\mathcal{N}_E^{\perp} = \mathcal{N}_{\R\setminus E}$, and $\mathcal{N}_E = \mathcal{N}_{E'}$ if and only if $E$ and $E'$ differ by a set of measure zero. So the reducing subspaces of $\mathcal{H}_{\infty}$ correspond exactly to the measurable subsets of $\R$ modulo null sets.
\end{enumerate}
\end{Remark*}

A vector $f \in L^2(0, \infty)$ is {\em $\ast$-cyclic }for $\mathcal{H}_{\infty}$ if the only {\em reducing} subspace containing $f$ is all of $L^2(0 ,\infty)$. Since $I - \mathcal{H}_{\infty}^{*}$ is unitarily equivalent to $M_{\phi}$ on $L^2(i \R)$, via $\mathcal{Q}$, we see that the $\ast$-cyclic vectors of $\mathcal{H}_{\infty}$ correspond to the $\ast$-cyclic vectors of $M_{\phi}$ and, from Wiener's theorem, these are the functions in $L^2(i \R)$ which do not vanish on any set of positive measure. We summarize this discussion with the following result.

\begin{Theorem}\label{prop:111}
For $f \in L^2(0, \infty)$ the following are equivalent.
\begin{enumerate}
\item[(a)] $f$ is a $\ast$-cyclic vector for $\mathcal{H}_{\infty}$.
\item[(b)] There is an $h \in L^2(i \R)$ that does not vanish on any set of positive measure such that
$$f(t) = \int_{-\infty}^{\infty} t^{-ix - \frac{1}{2}} h(i x)dx.$$
\item[(c)] The function $$\beta \mapsto \int_{-\infty}^{\infty} f(e^{x}) e^{(i \beta + \frac{1}{2}) x}dx$$ does not vanish on any  subset of $\R$ of positive measure.
\end{enumerate}
\end{Theorem}

Condition (c) of Theorem~\ref{prop:111} is the practical form of the result. The function in (c) is just $\sqrt{2\pi}\,(\mathcal{Q}f)(i\beta)$, the Mellin transform of $f$ on the line $\re s = \frac12$. So deciding whether $f$ is $\ast$-cyclic for $\mathcal{H}_{\infty}$ comes down to one concrete computation with $f$ itself, followed by a look at the zero set of the result. In particular, there is no need to identify the image of $f$ in $L^2(\T, m)$ under the unitary equivalence with the bilateral shift. This is exactly what makes the examples below and Corollary~\ref{supporttstsaaa} possible. In each case $\ast$-cyclicity is verified either by an explicit Fourier transform or by a Paley--Wiener argument applied directly to $Tf$.

Some specific  $\ast$-cyclic vectors for $\mathcal{H}_{\infty}$ can be derived as follows.

\begin{Example}
Consider the function
$$iy \mapsto \frac{1}{1-iy} \in L^2(i \R).$$
This function does not vanish on $i \R$ and thus is $\ast$-cyclic for $M_{\phi}$.
To calculate $\mathcal{Q}^{-1} = T^{-1} \mathcal{F}$ (see Proposition \ref{inversion lemma}) applied to this function,  we
observe that
\begin{equation}\label{eq:invfte-x}
 \int_0^\infty e^{-x} e^{iyx} \, dx = \frac{1}{1-iy},
 \end{equation}
 and thus
$$\mathcal{F}\left(\frac{1}{1-iy}\right)(x) = \sqrt{2 \pi} e^{-x} \chi_{(0, \infty)}(x) \in L^2(\R).$$
Now use the formula  $(T^{-1} g)(t) = g(\log t)/\sqrt{t}$ from \eqref{TTTTinverse} to see  that
$$\mathcal{Q}^{-1}\left(\frac{1}{1-iy} \right)(t) = \sqrt{2 \pi} \frac{1}{t^{\frac{3}{2}}} \chi_{(1, \infty)}(t).$$
Note the use of the identity $\chi_{(0, \infty)}(\log t) = \chi_{(1, \infty)}(t)$ in the above calculation.
Hence, an example of a $\ast$-cyclic vector for $\mathcal{H}_{\infty}$ is
$$ \frac{1}{t^{\frac{3}{2}}} \chi_{(1, \infty)}(t).$$

For $n = 1, 2, \ldots,$ by differentiating \eqref{eq:invfte-x} $(n-1)$ times with respect to $y$ we
see that
$$\mathcal{F}\Big(\frac{1}{(1-iy)^n}\Big)(x) = c_n e^{-x} x^{n - 1}\chi_{(0, \infty)}(x),$$
where $c_n$ is a nonzero constant depending only on $n$. Thus,
we obtain the following family of  $\ast$-cyclic vectors for $\mathcal{H}_{\infty}$:
$$\frac{1}{t^{\frac{3}{2}}} (\log t)^{n - 1} \chi_{(1, \infty)}(t), \quad n = 1, 2, \ldots$$
\end{Example}

\begin{Example}
One can also provide other $\ast$-cyclic vectors for $\mathcal{H}_{\infty}$ by considering functions in $L^2(i\R)$ defined by
$$G_{N}(i y) = e^{-\frac{y^2}{2}} H_{4 N}(y), \quad N = 0, 1, 2, \ldots,$$
where $H_{K}$ is the $K$-th Hermite polynomial. Since these functions only vanish at a finite number of points on $i \R$, they are $\ast$-cyclic vectors  for $M_{\phi}$ on $L^2(i \R)$. Moreover, the Fourier transform of $G_{N}$ is equal to itself (as these are eigenvectors for the Fourier transform corresponding to the eigenvalue $1$ \cite{MR4548181}).  Thus, $T^{-1} \mathcal{F} G_{N}$ will be a $\ast$-cyclic vector for the Hardy operator $\mathcal{H}_{\infty}$.
Note that
$$g_{N} : = T^{-1} \mathcal{F} G_{N}$$ can be computed as
$$g_{N}(t) = \frac{1}{\sqrt{t}} e^{-\frac{1}{2}(\log t)^2} H_{4N}(\log t).$$
For example, when $N = 0$ we get
$$g_{0}(t) = \frac{e^{-\frac{1}{2} \log ^2(t)}}{\sqrt{t}}$$
is a $\ast$-cyclic vector for $\mathcal{H}_{\infty}$.
When $N = 1$ we get
$$g_{1}(t) = \frac{4 e^{-\frac{1}{2} \log ^2(t)} \left(4 \log ^4(t)-12 \log ^2(t)+3\right)}{\sqrt{t}}$$
is a $\ast$-cyclic vector for $\mathcal{H}_{\infty}$.
\end{Example}

A wider class of $\ast$-cyclic vectors for $\mathcal{H}_{\infty}$ comes from this next result.

\begin{Corollary}\label{supporttstsaaa}
For each $a>0$ all nonzero $L^2(0,\infty)$ functions supported on $[0,a]$ are $*$-cyclic for $\mathcal{H}_{\infty}$.
Likewise, all nonzero functions supported on $[a,\infty)$ are also $*$-cyclic.
\end{Corollary}
\begin{proof}
Take an $f \in L^2(0, \infty) \setminus \{0\}$ supported on $[0,a]$.
By Theorem~\ref{prop:111} we wish to show that the function
\[
\beta \mapsto \int_{-\infty}^\infty f(e^x) e^{i\beta x + x/2} \, dx
\]
is nonzero almost everywhere on $\R$. By our assumption on $f$, this function becomes
\[
i\beta \mapsto\int_{-\infty}^{\log a} g(x) e^{i\beta x } \, dx
\]
where $g$ is the $L^2(\R)$ function defined by $g(x)=(Tf)(x)=f(e^x)e^{x/2}$.
Now set $y=\log a-x$ to see that the function in question is
\[
i\beta \mapsto e^{i\beta \log a}\int_{0}^\infty g(\log a-y)  e^{-i\beta y} \, dy.
\]
By the Paley--Wiener theorem, recall \eqref{LTT}, the integral above belongs to $H^{2}_{+}(i\R)$, and, via standard Hardy space theory,  is nonzero almost everywhere.

For the second part, the integral becomes
\[
i\beta \mapsto\int_{\log a}^{\infty} g(x) e^{i\beta x } \, dx
\]
and we write $y=x-\log a$ to obtain
\[
i\beta \mapsto e^{i \beta \log a}\int_{0}^{\infty} g(y+\log a) e^{i\beta y } \, dy,
\]
with the same conclusion, except that the function above belongs to $H^{2}_{-}(i\R)$.
\end{proof}

It follows immediately that all nontrivial functions  supported on
a finite subinterval $[a,b]$ of  $[0, \infty)$ are $\ast$-cyclic for $\mathcal{H}_{\infty}$.\\

%
%

\subsection{Non-reducing subspaces for $\mathcal{H}^{*}_{\infty}$}

With Theorem \ref{prop:111} at hand, some remarks about the non-reducing subspaces of $\mathcal{H}^{*}_{\infty}$ are in order. First, inverting $\mathcal{Q}$ from \eqref{Mellintrack}, we see that for $g \in L^2(i\R)$ we have
$$
(\mathcal{Q}^{-1} g)(t)  = \frac{1}{\sqrt{2 \pi}} \int_{-\infty}^{\infty} t^{-ix - \frac{1}{2}} g(i x) dx.
$$
By Helson's theorem \cite[Ch.~II]{MR0171178}, the non-reducing $M_{\phi}$-invariant subspaces are $q(i x) H^{2}_{+}(i \R)$, where $q$ is unimodular, and these can be written as
$$\left\{ix \mapsto  q(i x)\int_{0}^{\infty} e^{-ix w} h(w) dw: h \in L^2(0, \infty)\right\}.$$
Note the use of the Paley--Wiener theorem here. Putting this all together gives us the following.

\begin{Theorem}\label{thm:7.1}
For a subspace $\mathcal{M}$ of $L^2(0, \infty)$, the following are equivalent.
\begin{enumerate}
\item[(a)] $\mathcal{M}$ is a non-reducing invariant subspace for $\mathcal{H}_{\infty}^{*}$.
\item[(b)] There is a unimodular function $q$ on $i \R$ such that
\begin{equation}\label{eq:redhstar}
\mathcal{M} =  \left\{t \mapsto  \int_{-\infty}^{\infty} t^{-ix - \frac{1}{2}} q(i x) \left(\int_{0}^{\infty} e^{-ix w} h(w) dw\right) dx: h \in L^2(0, \infty)\right\}.
\end{equation}
\end{enumerate}
\end{Theorem}

\begin{Example}
For each $a > 0$, the subspace $\chi_{[a, \infty)} L^2(0, \infty)$ is
$\mathcal{H}_{\infty}$-invariant and thus its orthogonal complement
$\chi_{[0, a]} L^{2}(0, \infty)$ is $\mathcal{H}^{*}_{\infty}$-invariant (and non-reducing). The identity
$$\int_{0}^{a} f(t) t^{i \beta - \frac{1}{2}} dt = a^{i \beta} \sqrt{a} \int_{0}^{1} f(a u) u^{i \beta - \frac{1}{2}} du$$
shows that
$$\mathcal{Q} \chi_{[0, a]} L^{2}(0, \infty) =   H^{2}_{+}(i \R).$$
\end{Example}

To classify the non-reducing invariant subspaces for $\mathcal{H}_{\infty}$, rather than $\mathcal{H}_{\infty}^*$, we note, from the fact that $|q(ix)| = 1$ almost everywhere on $i\R$ and that $H^{2}_{-}(i \R) \perp H^{2}_{+}(i\R)$ (recall \eqref{***}), that
the orthogonal complement of the subspace $q(ix)H^2_+(i\R)$ is $ q(ix) H^2_-(i\R)$, also
that $f$ belongs to $H^2_+(i\R)$ if and only if $ix \mapsto f(-ix)$ belongs to $H^2_-(i\R)$ (recall \eqref{Ceeee}).
Therefore, the corresponding formula for the non-reducing invariant subspaces for $\mathcal{H}_{\infty}$ is obtained
by replacing \eqref{eq:redhstar} with
\begin{equation}\label{eq:nonreducingh}
\mathcal{M} =  \left\{t \mapsto  \int_{-\infty}^{\infty} t^{-ix - \frac{1}{2}} q(i x) \left(\int_{0}^{\infty} e^{ix w} h(w) dw\right) dx: h \in L^2(0, \infty)\right\}.
\end{equation}

Returning to Corollary \ref{supporttstsaaa}, evidently, a function defined on $[0,a]$ cannot be cyclic for $\mathcal{H}_{\infty}^*$ alone, as
its images under powers of $\mathcal{H}^{*}_{\infty}$ are all supported on $[0,a]$.
Curiously, it need not be cyclic for $\mathcal{H}_\infty$ either, as the following example shows.

\begin{Example}
Recall that $\chi$ denotes the characteristic function of $[0,1]$. A preliminary calculation
with Fourier transforms yields
\[
\frac{1}{\sqrt{2\pi}}\int_{-\infty}^0 e^{ixy} e^{\frac{y}{2}} \, dy = \frac{1}{\frac12+ix}
\]
and hence (inverting the transform)
\[
\frac{1}{\sqrt{2\pi}}\int_{-\infty}^\infty e^{-ixy} \frac{1}{ \frac12+ix} \, dx = e^{\frac{y}{2}} \chi_{(-\infty,0)}(y).
\]
Writing $t=e^y$ we have
\[
\frac{1}{\sqrt{2\pi}}\int_{-\infty}^\infty t^{-ix-\frac{1}{2}} \frac{1}{ \frac12+ix} \, dx =  \chi(t).
\]
Comparing this with \eqref{eq:nonreducingh}, we see that $\chi$ lies in the
nontrivial $\mathcal{H}_{\infty}$-invariant
subspace corresponding to the unimodular function
$$q(ix)=\frac{\frac12-ix}{\frac12+ix}$$
and in this case
we have $h(w)=\frac{1}{\sqrt{2 \pi}} e^{-\frac{w}{2}}$ in \eqref{eq:nonreducingh}.
\end{Example}

A remarkable theorem of Kargaev  \cite{MR648415,MR1206343} says that given any positive number $\ell$ there exists a measurable set $E \subset \R$ such that $|E| < \infty$ and $\mathcal{F} \chi_{E}$ is zero a.e.\ on $(-\ell, \ell)$. In particular, this yields the following corollary:

\begin{Corollary}
There exist nonzero functions in $ L^2(0, \infty)$ vanishing on an interval which are not $\ast$-cyclic vectors for $\mathcal{H}_{\infty}$.
\end{Corollary}

\begin{proof}
Let $\ell > 0$, and let $E$ be as in Kargaev's theorem. Then $\mathcal{F}\chi_E \in L^2(\R)$, since $|E| < \infty$. Put $f := T^{-1}\mathcal{F}\chi_E \in L^2(0,\infty)$, which is nonzero. By \eqref{TTTTinverse}, $f$ vanishes almost everywhere on $(e^{-\ell}, e^{\ell})$. Moreover, by Proposition~\ref{inversion lemma}, $(\mathcal{Q} f)(i\beta) = (\mathcal{F}^{-1}\mathcal{F}\chi_E)(\beta) = \chi_E(\beta)$. This vanishes on $\R \setminus E$, a set of infinite measure. By Theorem~\ref{prop:111}, $f$ is not $\ast$-cyclic for $\mathcal{H}_{\infty}$.
\end{proof}

\section{Cyclic vectors: examples}\label{cyclkcic5}

Beurling's theorem says that the cyclic vectors for $M_{\phi}$ on $H^{2}_{+}(i\R)$ are the outer functions in $H^{2}_{+}(i\R)$ \cite{Garnett}.
From Proposition \ref{98rgegbrfewrfvFF}, the cyclic vectors for $\mathcal{H}_1^*$  are those $f \in L^2(0, 1)$ such that the $H^2(\C^{+})$ function
$$(\mathcal{Q} f)(z) = \frac{1}{\sqrt{2\pi}}\int_{0}^{1} f(t) t^{z - \frac{1}{2}}dt, \quad z \in \C^{+},$$ is an outer function. Recall \eqref{sd98gsdf1111100} which shows that $\mathcal{Q} f$ is the Laplace transform of $$(W f)(u) = f(e^{-u}) e^{-\frac{u}{2}}.$$ In general, this seems difficult to unpack. However, we can provide a few examples.

\begin{Example}
A first class of examples of cyclic vectors for $\mathcal{H}_1^{*}$ is
$$f_{\alpha}(t)=t^{\alpha}, \quad \re\alpha >- \tfrac{1}{2}$$ (a necessary condition for $f_{\alpha}$ to belong to $L^2(0, 1)$). Here
$$(W f_{\alpha})(u)=e^{-\frac{u}{2}-\alpha u}$$  has (normalized) Laplace transform equal to $\frac{1}{\sqrt{2\pi}}(s+\alpha+\frac{1}{2})^{-1}$, which is outer.
\end{Example}

In order to provide examples of  cyclic vectors for $\mathcal{H}_{\infty}^{*}$, we first note that an invariant subspace for $\mathcal{H}_{\infty}^*$ has the form
\[
\mathcal{M} = \left\{t \mapsto \int_{-\infty}^{\infty} t^{- i x - \frac{1}{2}} f(ix) dx: f \in \mathcal{V}\right\}
\]
where $\mathcal V$ is either $\chi_{iE} L^2(i \R)$ or  $q(ix)H^2_+(i\R)$ for some unimodular $q$.


Writing $t=e^u$, with $-\infty < u < \infty$, we have $g \in \mathcal M$ if
\[
g(e^u)=e^{-\frac{u}{2}} \int_{-\infty}^\infty e^{-ixu} f(ix) \, dx.
\]

\begin{Example}\label{OOJNDHFSJ}
Here are two examples of cyclic vectors for $\mathcal{H}_{\infty}^*$:
\begin{enumerate}
\item[(a)] ${\displaystyle g(t)=\frac{1}{1 + t}}$;
\item[(b)] $g(t)= \dfrac{1}{\sqrt t(1+(\log t)^2)}$.
\end{enumerate}
To see this, first observe that
\[
e^{\frac{u}{2}}g(e^u) = \begin{cases} \dfrac{e^{\frac{u}{2}}}{e^u + 1} &  \text{in case (a),} \\ \\
 \dfrac{1}{1+u^2} & \text{in case (b)}.
\end{cases}
\]
In each case we can invert the Fourier transform explicitly to obtain
 \[
 f(ix) = \begin{cases} \sqrt{\dfrac{\pi}{2}} \sech \pi x &  \text{in case (a),} \\
 \sqrt{\dfrac{\pi}{2}}e^{-|x|} & \text{in case (b)}.
\end{cases}
\]
Each of these functions $f$ is nonzero everywhere, so $g$ cannot lie in a
nontrivial reducing subspace for $\mathcal{H}_{\infty}$, and likewise the $H^2_+(i\RR)$ condition
\begin{equation}\label{logInt}
\int_{-\infty}^{\infty} \frac{\log |f(iw)|}{1+w^2} dw > -\infty
\end{equation}
(a well known necessary condition for $f$ to be noncyclic for $M_{\phi}$ on $L^2(i\R)$ -- see for example \cite[Cor.~A.6.4.2]{MR1864396}) is {\em not} satisfied, so $g$ does not
lie in a (nontrivial)  non-reducing invariant subspace. Therefore,  both these $g$ are cyclic vectors for $\mathcal{H}_{\infty}^*$.


%
%
\end{Example}

The above example can be generalized to the following.
\begin{Example}\label{77cycliciciuccyclicicicc}
For  each $n = 1, 2, \ldots,$ the function
\[
g_n(t) =  \frac{1}{(t+1)^n}
\]
is a cyclic vector for $\mathcal{H}_{\infty}^*$.
This can be proved by induction. We have already seen the case $n=1$ from Example \ref{OOJNDHFSJ}. Let
$F_n$ denote the inverse Fourier transform
of $$\frac{e^{\frac{u}{2}}}{(e^u+1)^n}$$ so that $F_1(x)=\sqrt{\frac{\pi}{2}}\sech \pi x$.
Evaluating the integral in
\[
F_n(x)=\frac{1}{\sqrt{2\pi}} \int_{-\infty}^\infty e^{iux}\frac{e^{\frac{u}{2}}}{(e^u+1)^n} \, du
\]
 by parts we have that
\[
F_{n}(x)=-\frac{1}{ix}(-n(F_n(x)-F_{n+1}(x))+ \tfrac{1}{2}F_n(x))
\]
so that
$$F_{n+1}(x)=\frac{1}{n} (F_n(x)(-ix+n-\tfrac{1}{2})).$$
Thus, by induction, $F_{n+1}$ is a polynomial of degree $n$ in $x$ multiplied by $\sech \pi x$ and so
it satisfies the non-vanishing and non-log-integrability condition (recall \eqref{logInt}) for $g_n(t)$ to be a cyclic vector
for $\mathcal{H}_{\infty}^*$.
\end{Example}


\begin{Example}\label{Euldferefectsionforms}
One can also consider more rapidly decreasing functions such as
$$g(t) = e^{-t}$$
as candidates for cyclic vectors for $\mathcal{H}_{\infty}^{*}$.
Here
$$e^{\frac{u}{2}} g(e^u) = e^{\frac{u}{2}} e^{-e^u}$$
which has inverse Fourier transform equal to
$$f(ix)=\frac{\Gamma \left(\frac{1}{2}+i x\right)}{\sqrt{2 \pi }}$$
(just write out the integral formula for the Gamma function to see this Fourier transform detail; equivalently, $f = \mathcal{Q}g$ and 
$\int_0^\infty e^{-t}t^{ix-\frac12}\,dt = \Gamma(\frac12+ix).)$
This function    does not vanish on any set of positive measure and so $g$ is a $\ast$-cyclic vector for $\mathcal{H}_{\infty}$. To see that $g$ is a {\em cyclic} vector for $\mathcal{H}_{\infty}^{*}$, we use the Euler reflection formula \cite[p.~199]{MR510197}
$$\Gamma(z) \Gamma(1 - z) = \frac{\pi}{\sin \pi z}$$
with $z = \frac{1}{2} - i x$ to see that
$$\Gamma(\tfrac{1}{2}-ix) \Gamma(\tfrac{1}{2}+ix) =  \frac{\pi}{\sin \pi(\tfrac{1}{2}-ix)} = \frac{\pi}{\cosh(\pi x)}.$$

The two factors on the left are complex conjugates of each other and so
$|\Gamma(\frac{1}{2}+ix)|^2 = \pi/\cosh(\pi x)$. Hence $|\Gamma(\frac{1}{2}+ix)| \sim \sqrt{2\pi}\, e^{-\frac{\pi |x|}{2}}$ as $|x| \to \infty$, so $\log|f(iw)| = -\frac{\pi}{2}|w| + O(1)$ and the integral in \eqref{logInt} diverges to $-\infty$. This eliminates the possibility of log integrability from \eqref{logInt}.
 \end{Example}

Let us make the following observation about the cyclic vectors for $\mathcal{H}_{\infty}$ and those for $\mathcal{H}_{\infty}^{*}$.

\begin{Proposition}\label{share}
For $f \in L^2(0, \infty)$ the following are equivalent.
\begin{enumerate}
\item[(a)] $f$ is a cyclic vector for $\mathcal{H}_{\infty}$.
\item[(b)] $f $ is a cyclic vector for $\mathcal{H}_{\infty}^{*}$.
\end{enumerate}
\end{Proposition}

\begin{proof}
From our discussion from Proposition \ref{yyysysSYY} and \eqref{Ceeee}, recall that if $C: L^2(i \R) \to L^2(i \R)$ is defined by $(C g)(ix) = g(-ix)$, then $C$ is unitary (and is its own adjoint) and $C M_{\phi} C = M_{\overline{\phi}}$. Furthermore, $C H^2_{+}(i \R) = H^{2}_{-}(i \R)$.

Now observe that $f \in L^2(0, \infty)$ is cyclic for $\mathcal{H}_{\infty}^{*}$ if and only if $\mathcal{Q} f$ is cyclic for $M_{\phi}$ -- which is precisely the condition that $\mathcal{Q} f$ is nonvanishing almost everywhere on $i\R$ and that $\mathcal{Q} f \not \in q H^2_{+}(i\R)$ for any unimodular  function $q$ on $i \R$. The conditions for a function in $L^{2}(i \R)$ to be cyclic for $M_{\overline{\phi}}$ are that it is nonvanishing almost everywhere on $i \R$ and does not belong to any space $p H^{2}_{-}(i \R)$, where $p$ is unimodular on $i \R$.

Our earlier discussion  said  that $\mathcal{Q} f$ is cyclic for $M_{\phi}$ if and only if $C \mathcal{Q} f$ is cyclic for $M_{\overline{\phi}}$. But $\mathcal{Q} f$ is non-vanishing almost everywhere if and only if $C \mathcal{Q} f$ is non-vanishing almost everywhere and $\mathcal{Q} f$ does not belong to $q H^2_{+}(i \R)$ if and only if $C \mathcal{Q} f$ does not belong to $(C q) H^{2}_{-}(i\R)$. Thus, $f$ is cyclic for $\mathcal{H}_{\infty}^{*}$ if and only if $f$ is cyclic for $\mathcal{H}_{\infty}$.
To make the last step precise, note that for $g \in L^2(i\R)$ nonvanishing almost everywhere, $g \in qH^2_{+}(i\R)$ for some unimodular $q$ if and only if $\int_{-\infty}^{\infty} \frac{\log|g(iw)|}{1+w^2}\,dw > -\infty$ (for the converse, take $q = g/h$ with $h \in H^2_{+}(i\R)$ outer and $|h| = |g|$). The same holds with $H^2_{-}(i\R)$ in place of $H^2_{+}(i\R)$, since $w \mapsto -w$ preserves $dw/(1+w^2)$. Hence $M_{\phi}$ and $M_{\overline{\phi}}$ have the same cyclic vectors, and applying this to $g = \mathcal{Q} f$ gives the equivalence.
\end{proof}

We conclude that the cyclic vectors for $\mathcal{H}_{\infty}^{*}$ from Examples \ref{OOJNDHFSJ}, \ref{77cycliciciuccyclicicicc}, and \ref{Euldferefectsionforms} are also cyclic for $\mathcal{H}_{\infty}$.

\section{Invariant subspaces for $\mathcal{H}_{\infty}$ and semigroups}

Following the work of Arvanitidis and Siskakis \cite{MR3043050}, we consider  the weighted composition semigroup $\{C_t\}_{t\geq 0}$ on $L^2(0,\infty)$ defined by
\[
(C_t f)(x)=  e^{-\frac{t}{2}} f(e^{-t} x), \quad 0 \leq x < \infty.
\]
Then, using the substitution $y=e^{-t}x$, we obtain
\begin{align*}
\int_0^\infty (C_t f)(x) e^{-\frac{t}{2}} \, dt & = \int_0^\infty e^{-t}  f(e^{-t} x) \, dt\\
& = \frac{1}{x} \int_0^x f(y) \, dy \\
&= (\mathcal{H}_{\infty} f)(x).
\end{align*}
One can also verify  that $\|C_{t} f\| = \|f\|$ for all $f \in L^2(0, \infty)$ and thus $\{C_t\}_{t \geq 0}$ is a contractive semigroup (which will be important in a moment). The connection between $\{C_t\}_{t \geq 0}$ and the invariant subspaces for $\mathcal{H}_{\infty}$ is the following.

\begin{Proposition}
A subspace of $L^2(0, \infty)$ is invariant for $\mathcal{H}_{\infty}$ if and only if it is invariant under every $C_t$, $t \geq 0$.
\end{Proposition}

Let us outline the proof of this using some well known results about semigroups. Recall that for a given $C_0$-semigroup  $\{T(t)\}_{t\geq0}$ of bounded linear operators on a Hilbert space $H$, there exists a closed and densely defined linear transformation $A$  called the {\em generator} of $\{T(t)\}_{t\geq0}$,  that determines the semigroup. It is defined by
\begin{equation*}
A \vec{x}:=\lim_{t\to0^+}\frac{T(t) \vec{x}-\vec{x}}{t},
\end{equation*}
where the domain $\mathcal{D}(A)$ of $A$ consists of all $\vec{x}\in H$ for which this limit exists (see \cite[Chapter II]{MR2229872}, for instance). It is well-known that the generator is, in general, an unbounded operator which determines the $C_0$-semigroup.

Likewise, if $1$ belongs to the resolvent of $A$, namely, the set
$$\rho(A)=\{\lambda \in \mathbb{C}:\; (A-\lambda I): \mathcal{D}(A)\subset H\to H \mbox{ is bijective}\},$$
then, by the closed graph theorem,  $(A-I)^{-1}$ is a bounded operator on $H$ and the
Cayley transform of $A$, defined by
\begin{equation*}\label{eq:cogenerator}
\mathcal{V} =(A+I)(A-I)^{-1},
\end{equation*}
is a bounded operator on $H$ since $\mathcal{V}-I=2(A-I)^{-1}$. Note that $\mathcal{V}$ determines the semigroup uniquely, since $A$ does. This operator is called the \emph{cogenerator} of the $C_0$-semigroup $\{T(t)\}_{t\geq0}$. Observe that 1 is not an eigenvalue of $\mathcal{V}$.

In our setting, $T(t)$ is equal to $C_t$ on $L^2(0, \infty)$ and a straightforward computation yields that
$$Af(x)=-x f'(x)-\tfrac12 f(x)$$
for those $f\in L^2(0,\infty)$ for which $xf' \in L^2(0, \infty)$.
Likewise, since each $C_t$ is an isometry, $\{\lambda \in \C: \re\lambda > 0\} \subset \rho(A)$; in particular, $1, \frac12 \in \rho(A)$ and, having in mind that the resolvent can be expressed in terms of the Laplace transform of the semigroup (see \cite[Thm.~11, p.~622]{MR1009162}), we note that
\[
((A-\tfrac12 I)^{-1} f)(x)= -\int_0^\infty e^{-\frac{t}{2}} (C_{t}f)(x) \, dt=-(\mathcal{H}_{\infty} f)(x),  \quad 0 < x < \infty.
\]
Accordingly, if we write $\mathcal{W}=(2A+I)(2A-I)^{-1}$, the cogenerator of the semigroup $\{C_{2t}\}_{t \geq 0}$ we have that
$$\mathcal{W}-I = (A-\tfrac12 I)^{-1}=-{\mathcal H}_\infty,$$
and, using the fact that $\{C_{2t}\}_{t \geq 0}$ is a contractive semigroup (discussed earlier), the common invariant subspaces for the semigroup $\{C_t\}_{t\geq 0}$ are precisely those for ${\mathcal H}_\infty$ \cite[p.~146]{MR629828}.

This approach was used in \cite{MR4757014} in the context of a different semigroup
to obtain information about the invariant subspaces of the discrete Ces\`aro operator  on $H^2$ (see also \cite{GPR} for more on invariant subspaces of the Ces\`aro operator). Here
the fact that the semigroup $\{C_t\}_{t\geq 0}$ consists of dilations enables us
 to identify
some new invariant subspaces for ${\mathcal H}_\infty$.

\begin{Example}\label{bhsdf7gsdfds7f}
Consider the subspace $\mathcal{M}$  consisting of all  functions belonging to $L^2(0, \infty)$ that are constant almost everywhere on $[0, 1]$. One can see that $\mathcal{M}$ is the closed linear span of $\{C_t \chi: t \geq 0\}$, where $\chi$ is the characteristic function on $[0, 1]$. Clearly $C_{t} \mathcal{M} \subset \mathcal{M}$ for all $0 \leq t < \infty$ and one can also see that $\mathcal{M}$ is also $\mathcal{H}_{\infty}$-invariant.

\end{Example}

\begin{Example}\label{100ooOO00}
As a generalization of the previous example, take a fixed  $A>0$ along with a finite set of monomials
$x^{a_1}, \ldots, x^{a_n}$ with $\re a_k>-\tfrac{1}{2}$ for each $k$.
Let $\mathcal{M}$ be the subspace consisting of all $L^2(0,\infty)$
functions whose restriction to $[0,A]$ is a linear span of these monomials.
Note that $\mathcal{M}$ is invariant for all $C_t$ and for $\mathcal{H}_{\infty}$.
\end{Example}

In \eqref{TTTTisoisoms} we introduced the unitary operator $T: L^2(0,\infty) \to L^2(\R)$ defined by
\[
(Tf)(x)=f (e^x) e^{\frac{x}{2}}, \quad -\infty < x < \infty,
\]
with
\[
(T^{-1} g)(y)=\frac{g(\log y)}{\sqrt{y}}.
\]
Let us calculate $T C_t T^{-1}$ an operator on $L^2(\R)$.
For each $g \in L^2(\R)$ we have
\begin{align*}
(TC_t T^{-1} g)(x) &= ( C_tT^{-1} g)(e^x) e^{\frac{x}{2}} \\
&= (T^{-1}g)(e^{-t}e^x) e^{-\frac{t}{2}}e^{\frac{x}{2}} \\
&= \frac{g(x-t) }{e^{\frac{x}{2}}e^{-\frac{t}{2}}} e^{-\frac{t}{2}}e^{\frac{x}{2}}\\
&=   g(x-t),
\end{align*}
so that $C_t$ transforms into  a standard right shift on $L^2(\R)$.

By the Beurling--Lax theorem, the common invariant subspaces 
of $T C_t T^{-1}$ are functions with Fourier transform
$\chi_E L^2(\R)$ or $q H^2(U)$, where $U$ is the upper half-plane and $q$ is unimodular on $\R$. So we come back to the same conclusions by a
different route, since in $L^2(0,\infty)$ the invariant subspaces of ${\mathcal H}_\infty$ are of the form $T^{-1}\mathcal{M}$.




\section{The commutants}

In this section we provide explicit expressions for those bounded linear operators commuting with the Hardy operators $\mathcal{H}_1$ and $\mathcal{H}_{\infty}$.
Our description makes use of the identity
$$\Q(I-\mathcal{H}_{\infty}^*)\Q^{-1}=M_\varphi,$$ where $\Q=\F^{-1}T$ is unitary, so that
$$\Q(I-\mathcal{H}_\infty)\Q^{-1}= M^*_\varphi=M_{\bar\varphi}.$$
Recall our discussion from \S \ref{sectiosn333}. This, along with the well known fact \cite[p.~183]{MR4545809} that the commutant of $M_{\phi}$ on $L^2(i\R)$
consists of all multiplication operators $M_g$ with $g \in L^\infty(i\R)$, yields the following.
(As with Fourier transforms, we interpret the following as valid on a dense subspace,  then
extended by continuity.)

\begin{Proposition}\label{Commutant H infty}
A bounded linear operator $A$ in $L^2(0,\infty)$ commutes with $\mathcal{H}_{\infty}$ if and only if there exists a $g \in L^{\infty}(i \R)$ such that
\begin{equation}\label{AAAAllalalaaaa}
(Af)(t)=\frac{1}{{2\pi}}\int_{-\infty}^{\infty} t^{-iy-\frac{1}{2}}g(iy)\Big( \int_0^\infty   f(u) u^{iy-\frac{1}{2}}du\Big) dy
\end{equation}
for all $f \in L^2(0, \infty)$.
\end{Proposition}

\begin{proof}
If $A: L^2(0,\infty) \to L^2(0,\infty)$ commutes with $\mathcal{H}_{\infty}$,
we have
$$A(I-\mathcal{H}_{\infty})=(I-\mathcal{H}_{\infty})A$$ and so
\[
(\Q A \Q^{-1})(\Q (I-\mathcal{H}_{\infty})\Q^{-1})=(\Q (I-\mathcal{H}_{\infty})\Q^{-1})(\Q A \Q^{-1}),
\]
and hence we can write
\[
A= \Q^{-1}M_g \Q = T^{-1}\F M_g \F^{-1}T.
\]

We therefore have, for each $f \in L^2(0, \infty)$,
\begin{align*}
(Af)(t)&= t^{-\frac{1}{2}} (\F M_g \F^{-1}Tf) (\log t) \\
&= t^{-\frac{1}{2}}\frac{1}{\sqrt{2\pi}} \int_{-\infty}^{\infty} (M_g \F^{-1}Tf)(y)e^{-iy \log t} \, dy\\
&= \frac{1}{\sqrt{2\pi}} \int_{-\infty}^{\infty} t^{-iy-\frac{1}{2}}g(iy) (\F^{-1}Tf)(y) \, dy\\
&=  \frac{1}{{2\pi}}\int_{-\infty}^{\infty} t^{-iy-\frac{1}{2}}g(iy) \int_{-\infty}^{\infty} e^{iyx} (Tf)(x) \, dx \, dy\\
&= \frac{1}{{2\pi}}\int_{-\infty}^{\infty} t^{-iy-\frac{1}{2}}g(iy) \Big( \int_0^\infty   f(u) u^{iy-\frac{1}{2}}du\Big)dy.
\end{align*}
The above argument can be reversed and thus any $A$ from \eqref{AAAAllalalaaaa} belongs to the commutant of $\mathcal{H}_{\infty}$.
\end{proof}

Similar calculations can be made to find an expression for the commutant of ${\mathcal H_1}$. It is more convenient to work with the adjoint and use  \eqref{eq:h1equiv} in the form
\[
I-{\mathcal H}_1^* =\Q^{-1} M_\phi \Q.
\]
The commutant of $M_\phi$ on $H^2_+(i\R)$ consists of all multiplication operators $M_g$ with
$g \in H_{+}^\infty(i\R)$ (the boundary function of the bounded analytic functions on $\C^{+}$) \cite[p.~122]{MR4545809}, and thus, by \eqref{s0f8oigfjdkafdWW}, the commutant of ${\mathcal H}_1^*$ consists of
all operators
\begin{equation}\label{commutant}
B= \Q^{-1}M_g \Q = W^{-1}\Lap^{-1} M_g \Lap W.
\end{equation}
 Hence, similar as before, but for $f \in L^2(0, 1)$ we have

\begin{align*}
(Bf)(t)&= t^{-\frac{1}{2}} (\Lap^{-1} M_g \Lap Wf) (-\log t) \\
&= t^{-\frac{1}{2}}\frac{\sqrt{2\pi}}{ 2\pi i} \int_{\gamma-i\infty}^{\gamma+i\infty} (M_g \Lap Wf)(s)e^{-s \log t} \, ds\\
&=  \frac{1}{ 2\pi i} \int_{\gamma-i\infty}^{\gamma+i\infty} t^{-s-\frac{1}{2}}g(s)\int_0^\infty e^{-sx} (Wf)(x)   \, dx \, ds\\
&=  \frac{1}{ 2\pi i} \int_{\gamma-i\infty}^{\gamma+i\infty} t^{-s-\frac{1}{2}}g(s)\int_0^\infty e^{-sx-\frac{x}{2}} f(e^{-x})   \, dx \, ds\\
&=  \frac{1}{ 2\pi i} \int_{\gamma-i\infty}^{\gamma+i\infty} t^{-s-\frac{1}{2}}g(s)\int_0^1 f(u)   u^{s-\frac{1}{2}}   \, du \, ds\\
\end{align*}
where $\gamma>0$, as is usual in the definition of the inverse Laplace transform. We state the result as follows. (Once more, this is to be interpreted on a dense subspace and then extended by continuity.)

\begin{Proposition}\label{Commutant H 1}
A bounded linear operator $B$ in $L^2(0,1)$ commutes with $\mathcal{H}_{1}^*$ if and only if there exists a $g \in H_{+}^{\infty}(i\R)$ such that
$$(Bf)(t)= \frac{1}{ 2\pi i} \int_{\gamma-i\infty}^{\gamma+i\infty} t^{-s-\frac{1}{2}}g(s) \Big(\int_0^1 f(u)   u^{s-\frac{1}{2}} du\Big)ds
$$
for all $f \in L^2(0, 1)$,
where $\gamma>0$.
\end{Proposition}

Observe that a similar expression can be derived for the commutant of ${\mathcal H}_1$ by
taking adjoints,
although it now involves the conjugate Toeplitz operator $T_{\overline g}$ on $H^{2}(\C^{+})$.

\section{Frame vectors}\label{section frame}

\subsection{Basic notation}

For a bounded linear operator $T$ a separable Hilbert space $H$, a vector
$\vec{v} \in H$ is a {\em frame vector} for $T$ if the sequence $(T^{n} \vec{v})_{n \geq 0}$
forms a frame. In other words, if there exist constants $c_1,c_2 > 0$ such that
\begin{equation}\label{eq:defframe}
c_1 \|\vec{x}\|^2 \leq \sum_{n=0}^\infty |\langle \vec{x},T^n \vec{v}\rangle|^2 \leq c_2\|\vec{x}\|^2 \; \; \mbox{for all $\vec{x} \in H$}.
\end{equation}

The optimal $c_1, c_2$ above are called the {\em frame bounds}.
It is known (see   \cite[Chap.~4]{MR1228209}) that an equivalent condition for a sequence
$(\vec{v}_k)_{k \geq 0}$ to be a frame on a Hilbert space $H$  is that the mapping from $\ell^2(\N_{0})$ to $H$ defined by
\begin{equation}\label{eq:framemap}
 (a_n)_{n \geq 0} \mapsto \sum_{n=0}^\infty a_n \vec{v}_n
\end{equation}
is bounded and  surjective.

\subsection{Frame vectors for $S$ and $M_{\phi}$}Recently, Cabrelli, Molter, Paternostro, and Philipp  studied frame vectors for operators in \cite{MR4093918}. In particular,
\cite[Thm.~3.3]{MR4093918} asserts that for $T$ to possess a frame vector
it is  necessary that $T^*$ be similar to a strongly stable contraction, namely,
$$\lim_{n \to \infty} \|T^{*n}\vec{y}\| = 0 \; \; \mbox{ for all $\vec{y} \in H$.}$$
By the uniform boundedness principle, it follows that such $T$ must satisfy
$$\sup_{n \geq 0} \|T^n\| < \infty$$
and hence, by the spectral radius formula, $\sigma(T) \subseteq \overline{\D}$.
Thus, we cannot hope to find frame vectors for the Hardy operators $\mathcal{H}_1$ and $\mathcal{H}_{\infty}$, since, as discussed in the introduction, their spectra are not contained
in $\overline{\D}$.
However, also discussed in the introduction, we know that $I-\mathcal{H}_1^*$ is unitarily equivalent to the
unilateral shift $S$ on $H^2$ and $I-\mathcal{H}_{\infty}^*$ is unitarily equivalent to the bilateral shift on $L^2(\T)$.
Now $S$ on $H^2$ does possess frame vectors (as we shall demonstrate below), but $S^*$ does not, nor does the  bilateral shift on $L^2(\T)$. These negative results follow from the strong stability
criterion above but can also be proved directly, simply by noting that for such operators $T$ we cannot have
a uniform lower bound
\[
c_1 \|\vec{x}\|^2 \leq \sum_{n=0}^\infty |\langle T^{*n} \vec{x},  \vec{v}\rangle|^2
\]
if we let $\vec{x}$ range over functions $e_k$ with $e_k(z)=z^k$ (letting $k \to \infty$ for $S^*$, and $k \to -\infty$ for the bilateral shift).

The following result is a special case of \cite[Prop.~3.8]{MR4060361}. For the sake of completeness, we provide its short proof.

\begin{Proposition}\label{prop:frameforshift}
For the unilateral shift $(S f)(z) = z f(z)$ on $H^2$ the following are equivalent for $u \in H^2$:
\begin{enumerate}
\item[(a)] $u$ is a frame vector for $S$;
\item[(b)] $u$ and $1/u$ belong to $H^{\infty}$, the bounded analytic functions on $\D$.
\end{enumerate}
\end{Proposition}

\begin{proof}
Take $u \in H^2$, with $u(z)=\sum_{k \geq 0} u_k z^k$ and consider the condition described in \eqref{eq:framemap}.
Then
\[
\sum_{n = 0}^{\infty} a_n S^n u = \sum_{n=0}^\infty a_n S^n \Big(\sum_{k=0}^\infty u_k z^k\Big) = f(z) u(z),
\]
where $f(z)=\sum_{k \geq 0} a_k z^k  \in H^2$. Since $(a_n)_{n \geq 0} \in \ell^2$ (the sequence of Taylor coefficients of an $H^2$ function belong to $\ell^2 (\N_{0})$),
the mapping
$$(a_n)_{n \geq 0} \mapsto \sum_{n=0}^\infty a_n S^n u$$
from $\ell^2(\N_{0})$ to $H^2$
is bounded and surjective precisely when $uH^2=H^2$. By standard facts about multipliers of $H^2$ \cite[p.~121]{MR4545809}, this holds if and only if $u$ and $1/u$ belong to $H^{\infty}$, the bounded analytic functions on $\D$.
\end{proof}

If the sequence $(S^{n} u)_{n \geq 0}$ forms a frame, then it
even forms a Riesz basis, since with $u$ as in Proposition~\ref{prop:frameforshift}
the vectors are the image of the standard orthonormal basis $(z^n)_{n \geq 0}$ under
the bounded isomorphism $f \mapsto uf$.

It remains  for us to describe the frame vectors for $I-\mathcal{H}_1^*$. From our discussion from \S \ref{sectiosn333} we know that $M_\phi \mathcal{Q}=\mathcal{Q}(I-\mathcal{H}_1^*)$, where
$$\phi(\beta)=\frac{\beta-\frac{1}{2}}{\beta+\frac{1}{2}}$$ and $\mathcal{Q}: L^2(0, 1) \to H^{2}(\C^{+})$ is defined by
$$(\mathcal{Q} f)(z) = \frac{1}{\sqrt{2\pi}}\int_{0}^{1} t^{z - \frac{1}{2}} f(t) dt.$$
 Recall that we also have $\mathcal{Q}=\Lap W$ as in \eqref{WWWWWW}.
We shall use the transform from $\D$ onto the right half plane $\C^{+}$ defined by
\[
z=\frac{\beta-\frac12}{\beta+\frac12}, \quad \beta=\frac{1}{2} \frac{1+z}{1-z},
\]
for $z \in \D$ and $\beta \in \C^+$.
As we saw in \eqref{eq:Jh2h2}, an
 isomorphism  $J:H^2 \to H^2(\C^+)$ is given by
\[
(J f) (\beta)= \frac{1}{\sqrt{2 \pi}} \frac{1}{\beta+\frac{1}{2}} f \Big( \frac{\beta-\frac{1}{2}}{\beta+\frac{1}{2}} \Big), \quad f \in H^2,  \quad \beta \in \C^+,
\]
with inverse
\[
(J^{-1}F) (z)= \sqrt{2 \pi}  \frac{1}{1-z} F \Big( \tfrac12 \Big( \frac{1+z}{1-z}\Big)\Big), \quad F \in H^2(\C^+), \quad z \in \D,
\]
and we note that $J^{-1}M_\phi J= M_z = S$, the unilateral shift on $H^2$.

This yields the following result.
\begin{Corollary}\label{sdfgdsfg}
A function $F \in H^2(\C^{+})$ is a frame vector for $M_\phi$ if and only if $u:=J^{-1}F$
satisfies $u,1/u \in H^\infty(\D)$.
\end{Corollary}

\subsection{Frame vectors for $I - \mathcal{H}_{1}^{*}$} With Corollary \ref{sdfgdsfg} in hand, we can provide explicit examples of frame vectors for $I - \mathcal{H}_1^{*}$.

\begin{Example}\label{exmsxmxmapleee}
Consider the functions $t \mapsto t^\alpha$ which lie in  $L^2(0,1)$ when $\re \alpha> - \frac{1}{2}$. We claim
they are frame vectors for $I- \mathcal{H}_1^{*}$.
As discussed earlier, they correspond  via $\mathcal{Q}$ to
$$F_\alpha(\beta) =  \frac{1}{\sqrt{2\pi}}\frac{1}{\beta+\alpha+\frac{1}{2}} \in H^2(\C^+).$$
Now
\begin{align*}
J^{-1}F_\alpha(z) & = \frac{1}{1-z} \frac{1}{\frac12 \frac{1+z}{1-z} + \alpha + \frac12}\\
& = \frac{1}{\frac12(1+z)+(1-z)(\alpha+\frac12)}\\
&= \frac{1}{\alpha+1-\alpha z},
\end{align*}
which is an invertible function in $H^\infty$ (its reciprocal is a polynomial), since for $\alpha \neq 0$ its only pole $( \alpha + 1)/\alpha$
lies outside $\overline\D$ when $\re \alpha> -\frac{1}{2}$ (as $|\alpha+1| > |\alpha|$), while for $\alpha = 0$ it is constant.
Particular examples are $\alpha=0$ and $\alpha=1$, giving
$(J^{-1}F_\alpha) (z)=1$ and $(2 - z)^{-1}$, respectively (which are both clearly invertible elements of $H^{\infty}$).
\end{Example}

\begin{Example}\label{ex:mathica}
As a special case of the previous example, consider  the constant function $\chi \equiv 1$ on $[0, 1]$. This will correspond to the constant function $1$ on $\D$ (Example \ref{exmsxmxmapleee}). Thus, the sequence
$((I - \mathcal{H}_{1}^{*})^{n} \chi)_{n \geq 0}$ is a frame for $L^2(0, 1)$.
In fact
\begin{equation}\label{eq:IHn}
((I - \mathcal{H}_1^{*})^n \chi)(x) = L_{n}(-\log x), \quad n \geq 0,\quad 0 < x \leq 1,
\end{equation}
where $L_n$ is the $n$th Laguerre polynomial discussed earlier. Clearly this is true when $n=0$. Now, assuming the result for $n=k$ we have
\[
((I - \mathcal{H}_1^{*})^{k+1}  \chi)(x)=L_k(-\log x)-\int_x^1 \frac{L_k(-\log t)}{t} \, dt
\]
or, setting $w=-\log x$ and $v=-\log t$,
\[
((I - \mathcal{H}_1^{*})^{k+1} \chi)(e^{-w})=L_k(w)-\int_0^w L_k(v) \, dv=L_{k+1}(w)
\]
by an identity equivalent to \eqref{eq:lagshiftidentity}. Thus, we have \eqref{eq:IHn} by induction, and
%
an integral substitution shows that the
functions $(L_n(-\log x))_{n \geq 0}$  are an orthonormal basis for $L^2(0, 1)$.
\end{Example}

\subsection{Frame vectors and the commutant}

As it turns out from \cite[Proposition 3.8]{MR4060361}, one can, in a way, obtain {\em all} frame vectors from a single one. The general result for a bounded Hilbert space operator $T$ on $H$ is the following: If $\vec{x}_0$ is a particular  frame vector for $T$, then any frame vector $\vec{x}$ for $T$ must be of the form $V \vec{x}_0$, where $V$ is a bounded invertible operator on $H$ that also belongs to the commutant of $T$. We thank Victor Bailey for pointing this paper out  to us.

The commutant of the unilateral shift is $\psi(S)$, where $\psi \in H^{\infty}$ (which can also be written as $T_{\psi} f = \psi f$ \cite[p.~122]{MR4545809}) and $\psi(S)$ is invertible if and only if $1/\psi \in H^{\infty}$.
This means that the commutant of $I - \mathcal{H}_1^{*}$ is $\psi(I - \mathcal{H}_1^{*})$, where $\psi \in H^{\infty}$. Since we know that $\chi  \equiv 1$ is a frame vector for $I - \mathcal{H}_1^{*}$ (Example \ref{ex:mathica}) we have the following corollary. (Alternatively, the corollary follows directly from Proposition~\ref{prop:frameforshift}, since the unitary map $J^{-1}\mathcal{Q}: L^2(0,1) \to H^2$ intertwines $I-\mathcal{H}_1^{*}$ with $S$ and sends $\chi$ to the constant function $1$.)

\begin{Corollary}\label{generalframeHH}
An $f \in L^2(0, 1)$ is a frame vector for $I - \mathcal{H}_1^{*}$ if and only if there is a $\psi \in H^{\infty}$ with $1/\psi \in H^{\infty}$ such that $f = \psi(I - \mathcal{H}_1^{*}) \chi$.
\end{Corollary}

Note that an expression for $\psi(I - \mathcal{H}_1^{*})$ can be obtained from
Proposition~\ref{98rgegbrfewrfvFF} and in particular
\eqref{eq:h1equiv}, which implies an explicit unitary equivalence with $M_{\psi \circ \phi}$
on $H^2_{+}(i\R)$.



\subsection{Lack of frame vectors for the discrete Ces\`{a}ro operator}

We conclude by showing how the ideas above can be used to analyze the
discrete Ces\`{a}ro operator, which is defined on $H^2$ by
\[
(\mathcal{C}f) (z) = \frac{1}{z} \int_0^z \frac{f(\xi)}{1-\xi} \, d\xi,
\]
or equivalently, via power series, by
$$\mathcal{C} \Big(\sum_{n = 0}^{\infty} a_n z^n\Big)  = a_0 + \frac{a_0 + a_1}{2}  z+ \frac{a_0 + a_1 + a_2}{3} z^2 + \cdots.$$
For more information about the discrete Ces\`{a}ro operator, see, for example \cite{MR4740839} and the references therein.
 Work of Kriete and Trutt
\cite{MR281025,MR350489} produced a reproducing kernel Hilbert space $\mathfrak{H}$ of analytic functions on $\D$ that contains the polynomials as a dense set and such that
$I - \mathcal{C}$ is unitarily equivalent to  $M_z$ (multiplication by $z$) on
$\mathfrak{H}$.

The same argument used to classify the frame vectors for the shift $S$ on  $H^2$ (Proposition \ref{prop:frameforshift}), along with the fact from  \cite[Thm.~4]{MR350489} that $H^2 \subset \mathfrak{H}$ with continuous inclusion,   says that $u \in \mathfrak{H}$ will be a frame vector for $M_z$ on $\mathfrak{H}$ if and only if $u H^2 = \mathfrak{H}$. This will never happen as can be seen by the following argument: If $u H^2 = \mathfrak{H}$, then, since the constant functions belong to $\mathfrak{H}$, it follows that the function  $u$ has no zeros in $\D$. This means that the zero sequences for $u H^2$ functions (sequences $A \subset \D$ for which there is an $f \in u H^2 \setminus \{0\}$ with $f|_A = 0$) must also be zero sequences for $H^2$. Classical Hardy space theory says that  the zero sequences for $H^2$ are the  so-called Blaschke sequences.  However, from  \cite[\S4]{MR350489} there are nonzero functions in $\mathfrak{H}$ that vanish on non-Blaschke sequences. Since there are no frame vectors for $M_{z}$ on $\mathfrak{H}$, there are no frame vectors for $I - \mathcal{C}$.

\bibliographystyle{plain}

\bibliography{references}

\end{document}